\documentclass[12pt]{amsart}

\usepackage{amssymb,amsbsy}
\usepackage{latexsym,color,longtable}
\usepackage{verbatim,euscript}
\usepackage{fullpage}

\numberwithin{equation}{section}
\usepackage[colorlinks=true,linkcolor=blue,urlcolor=black,citecolor=magenta]{hyperref}

\input {cyracc.def}
\font\tencyr=wncyr10 %scaled \magstephalf
\font\tencyi=wncyi10 %scaled \magstephalf
\font\tencysc=wncysc10 %scaled \magstephalf
\def\rus{\tencyr\cyracc}
\def\rusi{\tencyi\cyracc}
\def\rusc{\tencysc\cyracc}

\newenvironment{proof*}
{\noindent {\sl Proof.}\quad }{\hfill
$\square$}

\catcode`\@=\active
\catcode`\@=11

\renewcommand{\@cite}[2]{[{{\bf #1}\if@tempswa , #2\fi}]}
\renewcommand{\@biblabel}[1]{[{\bf #1}]\hfill}
\catcode`\@=12

\newtheorem{thm}{Theorem}[section]
\newtheorem{lm}[thm]{Lemma}%[chapter]
\newtheorem{prop}[thm]{Proposition}%[chapter]
\newtheorem{conj}[thm]{Conjecture}

\theoremstyle{remark}
\newtheorem{rmk}[thm]{Remark}

\theoremstyle{definition}

\newtheorem{df}{Definition}
\newtheorem*{rema}{Remark}

\newenvironment{E6}[6]{%
{\small\begin{tabular}{@{}c@{}}
{#1}--{#2}--\lower3.5ex\vbox{\hbox{{#3}\rule{0ex}{2.5ex}}
\hbox{\hspace{0.4ex}\rule{.1ex}{1ex}\rule{0ex}{1.4ex}}\hbox{{#6}\strut}}--{#4}--{#5}
\end{tabular}}}

\newenvironment{E7}[7]{%
{\small\begin{tabular}{@{}c@{}}
{#1}--{#2}--{#3}--\lower3.5ex\vbox{\hbox{{#4}\rule{0ex}{2.5ex}}
\hbox{\hspace{0.4ex}\rule{.1ex}{1ex}\rule{0ex}{1.4ex}}\hbox{{#7}\strut}}--{#5}--{#6}
\end{tabular}}}

\newenvironment{E8}[8]{%
{\small\begin{tabular}{@{}c@{}}
{#1}--{#2}--{#3}--{#4}--\lower3.5ex\vbox{\hbox{{#5}\rule{0ex}{2.5ex}}
\hbox{\hspace{0.4ex}\rule{.1ex}{1ex}\rule{0ex}{1.4ex}}\hbox{{#8}\strut}}--{#6}--{#7}
\end{tabular}}}

\newcommand {\ah}{{\mathfrak a}}

\newcommand {\g}{{\mathfrak g}}
\newcommand {\h}{{\mathfrak h}}

\newcommand {\el}{{\mathfrak l}}

\newcommand {\es}{{\mathfrak s}}
\newcommand {\te}{{\mathfrak t}}
\newcommand {\ut}{{\mathfrak u}}
\newcommand {\z}{{\mathfrak z}}

\newcommand {\slv}{{\mathfrak {sl}}(\BV)}
\newcommand {\glv}{{\mathfrak {gl}}(\BV)}

\newcommand {\spv}{{\mathfrak {sp}}(\BV)}

\newcommand {\sov}{{\mathfrak {so}}(\BV)}

\newcommand {\esi}{\varepsilon}
\newcommand {\ap}{\alpha}

\newcommand {\lb}{\lambda}

\newcommand {\ca}{{\mathcal A}}

\newcommand {\N}{{\mathcal N}}
\newcommand {\co}{{\mathcal O}}

\newcommand {\BV}{{\mathsf{V}}}

\newcommand {\BZ}{{\mathbb Z}}
\newcommand {\BN}{{\mathbb N}}

\newcommand {\BQ}{{\mathbb Q}}

\newcommand {\ad}{{\mathrm{ad\,}}}

\newcommand {\codim}{{\mathrm{codim\,}}}

\newcommand {\hot}{{\mathsf{ht}}}
\newcommand {\htt}{\tilde{\textsl{ht\,}}}
\newcommand {\ind}{{\mathsf{ind}}}
\newcommand {\Lie}{{\mathrm{Lie\,}}}
\newcommand {\Ker}{{\mathsf{Ker\,}}}
\newcommand {\Ima}{{\mathsf{Im\,}}}

\newcommand {\rk}{{\mathsf{rk}}}
\newcommand {\spe}{{\mathsf{Spec}}}

\newcommand {\tri}{{\mathfrak{sl}}_2}
\newcommand {\sltri}{{\mathfrak{sl}}_3}
\newcommand {\GR}[2]{{\textrm{{\color{blue}\bf #1}}}_{#2}}
\newcommand {\GRt}[2]{{\color{blue}{\widetilde{\textrm{\bf #1}}}   }_{#2}   }

\newcommand {\ov}{\overline}

\newcommand {\beq}{\begin{equation}}
\newcommand {\eeq}{\end{equation}}
\newcommand {\eus}{\EuScript}
\newcommand {\ede}{\eus D(e)}
\newcommand {\edva}{e^{\langle 2\rangle}}
\newcommand {\etdva}{\tilde e^{\langle 2\rangle}}
\newcommand {\ltr}{{\langle 2\rangle}}

\renewcommand{\le}{\leqslant}
\renewcommand{\ge}{\geqslant}

\newfam\Bbbfam%
\font\Bbbfont=msbm10 scaled 1200%
\font\Bbbsmallfont=msbm8%
\textfont\Bbbfam=\Bbbfont\scriptfont\Bbbfam=\Bbbsmallfont%%
\def\varnothing{\hbox {\Bbbfont\char'077}}
\def\bbk{\hbox {\Bbbfont\char'174}}

\begin{document}
\setlength{\parskip}{2pt plus 4pt minus 0pt}
\hfill {\scriptsize March 30, 2010} 
\vskip1.5ex

\title[Divisible weighted Dynkin diagrams and reachable elements]
{On divisible weighted Dynkin diagrams and reachable elements}
\author[D.\,Panyushev]{Dmitri I. Panyushev}
\address{Independent University of Moscow, 
Bol'shoi Vlasevskii per. 11, 119002 Moscow, Russia
\hfil\break\indent
Institute for Information Transmission Problems, B. Karetnyi per. 19, Moscow 
127994}
\email{panyush@mccme.ru}
\urladdr{\url{http://www.mccme.ru/~panyush}}
\maketitle

%%%%%%%%%%%%%%%%%%%%%%
\section*{Introduction}

\noindent
Let $G$ be a connected simple algebraic group with Lie algebra $\g$  
and $e\in \g$ a nilpotent element. 
By the Morozov-Jacobson theorem,
there is an $\tri$-triple containing $e$, say $\{e,h,f\}$.
The semisimple element $h\in\g$ is called a {\it characteristic\/} of $e$.
Let $\eus D(e)$ be the weighted Dyn\-kin diagram of (the $G$-orbit of) $e$.
As is well known, the numbers occurring in this diagram belong to the set
$\{0,1,2\}$ (see Section~\ref{sect:prelim}  for details.). 
Suppose that $e$ is  {\it even}, which means that "1" does not occur in $\ede$. 
%Similarly, if "2" does not occur in $\eus D(e)$, then we say that $e$ is {\it odd\/}. 
Then one may formally divide $\ede$ by 2, i.e., replace all "2" in $\ede$ 
with "1". The resulting diagram, denoted  $\frac{1}{2}\ede$,
still looks like a weighted Dynkin diagram, and we
are interested in the following situation: 
\begin{itemize}
\item[$\framebox{\checkmark}$] \ 
Both  $\ede$ and $\frac{1}{2}\ede$  are weighted Dynkin diagrams; \quad
{\em equivalently,} 
\item[$\framebox{\checkmark}$] \ Both $h$ and $h/2$ are characteristics of 
nilpotent elements.
\end{itemize}

\noindent 
If such a division produces another nilpotent element, then one may expect that the corresponding  orbits have some interesting properties.

\begin{df}
A weighted Dynkin diagram $\ede$ or the corresponding nilpotent $G$-orbit $\co=G{\cdot}e$ 
is said to be {\it divisible\/} if $\frac{1}{2}\ede$ is again a weighted Dynkin diagram. 
For a divisible $\ede$, the pair of orbits corresponding to $\ede$ and $\frac{1}{2}\ede$ 
is called a {\it friendly pair}. 
\end{df}
\noindent
The orbit corresponding to $\frac{1}{2}\ede$ is denoted by  $\co^\ltr$, and we write 
$\edva$ for an element of $\co^\ltr$ with characteristic $h/2$.
Our goal is to classify the friendly pairs of nilpotent orbits 
for all simple Lie algebras and explore some of their properties.  
Write  $\g^x$ for the centraliser of $x\in\g$.

In Section~\ref{sect:gen-prop}, 
we prove that $\dim \g^{\edva}=\dim\Ker(\ad e)^2=\dim\g^e+\dim \g^e_{nil}$, where
$\g^e_{nil}$ is the nilpotent radical of $\g^e$; we also note that  if $e$ is divisible, then the 
{\it Dynkin index\/} of the simple 3-dimensional subalgebra $\textsf{span}\{e,h,f\}$ is 
divisible by 4. For the classical Lie algebras,
we characterise the partitions corresponding to the divisible orbits (Theorem~\ref{sl+sp+so})
and provide an explicit construction of $\edva$ via the  Jordan normal form of $e$. 
For instance, if $\g=\slv$ or $\spv$, then the partition 
$(\lb_1,\lb_2,\dots)$ of $\dim\BV$ gives rise
to a divisible orbit if and only if all $\lb_i$ are odd. Furthermore, if 
$SL(\BV){\cdot}e\subset \slv$
is divisible, then one can take $\edva=e^2$, which explains our notation.
For the exceptional Lie algebras, we merely provide a list of friendly pairs 
(Table~\ref{table_E}). 
Let $\el$ be a minimal  Levi subalgebra of $\g$ meeting $G{\cdot}e$.
Using our classification, we prove that  $L{\cdot}e$ is divisible (in $\el$) if and only if 
$G{\cdot}e$ is divisible (Theorem~\ref{thm:divis-Levi}).

For a divisible $\co=G{\cdot}e$, we  assume that $[h,\edva]=4\edva$.
The pair $(\co,\co^{\ltr})$ is said to be {\it very friendly}, if $\edva$ can additionally 
be chosen such that  $[e,\edva]=0$. In Section~\ref{sect:very},
we prove that all friendly pairs in the classical algebras are very friendly, whereas for the exceptional algebras there is only one exception (for $\g$ of type $\GR{F}{4}$).

The two nonzero nilpotent orbits in $\sltri$ represent the simplest example
of a friendly pair. Motivated by this observation,
we say that two nilpotent orbits $\tilde\co,\co \subset \g$ form an $\GR{A}{2}$-{\it pair},
if there is a subalgebra $\sltri\subset \g$ such that 
$\tilde\co\cap\sltri$ (resp. $\co\cap\sltri$) is the principal (resp. minimal) 
nilpotent orbit in $\sltri$.  Such an orbit $\co$ is called a {\it low-$\GR{A}{2}$ orbit}.
Every $\GR{A}{2}$-pair is friendly (with $\tilde\co$ divisible and $\co=\tilde\co^\ltr$), 
but not vice versa. %the converse is not always true. 
For $e\in\co$, we have
$e\in [\g^e,\g^e]$ (because this holds inside $\sltri$). Nilpotent elements (orbits) with this
property are said to be {\it reachable\/}. They have already been studied in \cite{EG, reach}. 
Let $\g^e=\bigoplus_{i\ge 0}\g^e(i)$ be the grading of $\g^e$ determined by a characteristic
of $e$.
For $e$ lying in a low-$\GR{A}{2}$ orbit,  we prove that  $\g^e$ is generated
by the Levi subalgebra $\g^e(0)$ and two elements in $\g^e(1)\subset \g^e_{nil}$
%, the nilpotent radical of $\g^e$
(Theorem~\ref{thm:reach}). In particular,
$\g^e_{nil}\subset [\g^e,\g^e]$  and $\g^e_{nil}$ is generated by the subspace $\g^e(1)$.
The latter provides a partial answer to \cite[Question\,4.6]{reach}, see also 
Section~\ref{sect:A2}. Theorem~\ref{thm:reach} can be regarded as an application 
(in case of $G=SL_3$)
of a general result that describes the structure of the space of $(U,U)$-invariants 
for any simple $G$-module~\cite[Theorem\,1.6]{odno-sv}. Here $U$ is a 
maximal unipotent subgroup of $G$.
For $\g$  exceptional, we derive  the list of $\GR{A}{2}$-pairs from results of 
Dynkin~\cite{dy}.

{\bf Acknowledgements.} {\small  This work was completed during my 
stay at the Max-Planck-Institut f\"ur Mathematik (Bonn). 
I am grateful to this institution for the warm hospitality and support.}

%%%%%%%%%%%%%%%%%%%%%%%%%%%%%%
\section{$\tri$-triples and centralisers}
\label{sect:prelim}

\noindent
We collect  some basic  facts on $\tri$-triples, associated 
$\BZ$-gradings, and centralisers of nilpotent elements.

Let $\g$ be a simple Lie algebra with a fixed triangular decomposition
$\g=\ut_-\oplus\te\oplus\ut_+$ and $\Delta$ the root system of $(\g,\te)$.
The roots of $\ut_+$ are positive. Write $\Delta_+$ (resp. $\Pi$) for
the set of positive (resp. simple) roots; $\theta$ is the highest root in $\Delta^+$. 
For $\gamma\in\Delta$, $\g_\gamma$ is the corresponding root space.
The Killing form on $\g$ is denoted by $\eus K$, and the induced bilinear form
on $\te^*_{\BQ}$ is denoted by $(\ ,\ )$.
For $x\in\g$, let $G^x$ and $\g^x$ denote its
centralisers in $G$ and $\g$, respectively. 
Let $\N \subset\g$ be
the cone of nilpotent elements. By the Morozov-Jacobson theorem %\cite[Ch.\,3]{CM}, 
each nonzero element $e\in\N$ can be
included in an $\tri$-triple $\{e,h,f\}$ (i.e.,
$[e,f]=h,\ [h,e]=2e,\ [h,f]=-2f$). The semisimple element $h$, which is called
a {\it characteristic} of $e$, determines the $\BZ$-grading of $\g$:
\[
 \g=\bigoplus_{i\in\BZ}\g( i)  \ ,
\]
where $\g( i) =\{\,x\in\g\mid [h,x]=ix\,\}$.
Set $\g({\ge}j)=\oplus_{i\ge j}\g( i) $.
%Since all characteristics of $e$ are $G^e$-conjugate, the properties of this 
%$\BZ$-grading do not depend on a particular choice of $h$.
%\\[.5ex]
The orbit $G{\cdot}h$ contains a unique element $h_+$ such that
$h_+\in\te$ and $\alpha(h_+)\ge 0$
for all $\alpha\in\Pi$.
The Dynkin diagram of $\g$ equipped with the numerical
marks $\ap(h_+)$, $\ap\in\Pi$, at the corresponding nodes
is called the
{\it weighted Dynkin diagram} of (the $G$-orbit of) $e$, denoted $\ede$. 
It is known that \par
(a) \ \cite[Theorem\,8.2]{dy} $\tri$-triples $\{e,h,f\}$ and $\{e',h',f'\}$
are $G$-conjugate {\sl if and only if\/}
$h$ and $h'$ are $G$-conjugate {\sl if and only if\/}
$\ede=\eus D(e')$; \par 
(b) \ \cite[Theorem\,8.3]{dy} $\ap(h_+)\in\{0,1,2\}$; \par
(c) \ \cite[Corollary\,3.7]{ko59} $\tri$-triples $\{e,h,f\}$ and $\{e',h',f'\}$
are $G$-conjugate {\sl if and only if\/} $e$ and $e'$ are $G$-conjugate.
\\[.6ex]
Let $G(0)$ (resp. $P$) denote the connected subgroup of $G$ with Lie algebra
$\g( 0)$ (resp. $\g({\ge} 0)$). 
Set $K = G^e\cap G(0)$.
The following facts on the structure of centralisers $G^e\subset G$ and
$\g^e\subset\g$ are standard, see \cite[ch.\,III, \S\,4]{ss} or \cite[Ch.\,3]{CM}.
\begin{prop}      \label{stab}
Let $\{e,h,f\}$ be an $\tri$-triple. Then
\begin{itemize}
\item[\sf (i)] \ $K=G^e\cap G^f$, and it is a maximal
reductive subgroup in both $G^e$ and $G^f$; \
$G^e\subset P$;
\item[\sf (ii)] \ the Lie algebra $\g^e$ is non-negatively
graded: $\g^e=\bigoplus_{i\ge 0} \g^e( i) $, where $\g^e( i) =\g^e\cap\g( i) $. 
Here $\g^e_{nil}:=\g^e({\ge}1)$ is the nilpotent  radical 
and $\g^e_{red}:=\g^e(0)$ is a Levi subalgebra of $\g^e$;
\item[\sf (iii)] \ 
$\ad e:\g( i-2)\to\g ( i) $ is injective
for $i\le 1$ and surjective for $i\ge 1$;
\item[\sf (iv)] \ $(\ad e)^i : \g(-i) \to\g( i) $ 
is one-to-one;
\item[\sf (v)] \ $\dim\g^e=\dim \g( 0)+\dim \g( 1) $ and\/
$\dim\g^e_{nil}=\dim \g( 1)+\dim \g( 2) $.
\end{itemize}
\end{prop}
\noindent
The {\it height\/} of $e\in\N$, denoted $\htt(e)$, is the 
maximal integer $m$ such that $(\ad e)^m\ne 0$. By Proposition~\ref{stab}(iii), we also have
$\htt(e)=\max\{i\mid \g(i)\ne 0\}$.
If $e$ is even, then $\htt(e)$ 
is even, but the converse is not true. If $l_\ap=\ap(h_+)$, $\ap\in\Pi$, 
are the numerical marks of $\ede$ and $\theta=\sum_{\ap\in\Pi} n_\ap \ap$, then
\begin{equation}  \label{height}
\htt(e)=\theta(h_+)=\sum_{\ap\in\Pi} l_\ap n_\ap .
\end{equation}

\noindent
{\bf Warning}. We will consider two notions of height: the above height of $e\in \N$ and the usual height
of a root $\nu\in\Delta$, denoted $\hot(\nu)$.

%%%%%%%%%%%%%%%%%%%%%%%%%%%%%%
\section{First properties of divisible orbits}
\label{sect:gen-prop} 

\noindent
We fix an $\tri$-triple $\{e,h,f\}$ containing $e\in\N$  
and work with the corresponding $\BZ$-grading of $\g$.
Recall that $e$ is {\it even\/} if and only if $\g(i)=0$ for $i$ odd.
Then the integer $\htt(e)$ is also even.
%Whenever the divisibility of $\ede$ is discussed, we assume that $e$ is even.
If $\co=G{\cdot}e$ is divisible, then the orbit corresponding to $\frac{1}{2}\ede$
is denoted by $\co^{\langle 2\rangle}$ and we write
$\edva$ for an  element of $\co^{\langle 2\rangle}$ with characteristic $h/2$.
That is, we assume that $\edva\in\g(4)$ and there is an 
$\tri$-triple of the form $\{\edva,h/2, f^{\langle 2\rangle}\}$.
%, and we may (and will) assume that $\edva\in\g( 4)$. 
By a result of Vinberg, $G(0)$ has finitely many orbits in $\g( i)$ for each  $i\ne 0$.
Our first observation is

\begin{prop}  \label{trivial} Suppose that $e\in\N$ is even. \\
1.  Let $\co'$ be the dense $G(0)$-orbit in $\g(4)$. % and $x\in\co'$. 
Then 
$\ede$ is divisible if and only if \ 
$\eus K(\g^x(0),h)=0$  for some (=\,any) $x\in\co'$. 
In this case,  any element of  $\co'$ can be taken  as $\edva$.

2. For any $x\in\g(4)$, one has $\htt(x)\le \frac{1}{2}\htt(e)$. 
If \ $\ede$ is divisible, then\/ $\htt(\edva) =\frac{1}{2}\htt(e)$.
\end{prop}
\begin{proof}
1. The condition $\eus K(\g^x(0),h)=0$ is equivalent to that $h\in\Ima(\ad x)$.
The rest  is clear. 

2. The first assertion is obvious; the second follows from \eqref{height}.
\end{proof}

\begin{rmk}
Using the ``support method'' for
nilpotent elements \cite[\S\,5]{t55}, we can prove that if $G{\cdot}e$ is not divisible
and $h'$ is a characteristic of a nonzero $x\in\g(4)$, then $\| h'\| < \frac{1}{2}\|h\|$. 
Still, it can happen that $\htt(x)= \frac{1}{2}\htt(e)$
for a generic $x\in\g(4)$. (For instance, consider the non-divisible even
orbit $\GR{A}{2}+3\GR{A}{1}$ for  $\g=\GR{E}{7}$. Here $\htt(e)=4$ and, obviously,
$\htt(x)\ge 2$ for all nonzero $x\in\N$.) 

\end{rmk}
\begin{prop}   \label{first}
1)  For any nonzero $e\in \N$, we have \\   \indent
$\dim\Ker(\ad e)^2=\dim\g( 0)+2\dim\g(1)+\dim \g( 2)=\dim\g^e+\dim\g^e_{nil}$.

\noindent
2) If \ $\ede$ is divisible, then 
\[
   \dim\g^{\edva}=\dim\Ker(\ad e)^2=\dim\g^e+\dim\g^e_{nil}\, .
\]
In particular, $\dim\g^e_{nil}$ is even.
\end{prop}\begin{proof}
1) This follows from Proposition~\ref{stab}(iii)--(v).

2)  Now $e$ is even, hence  $\g(1)=0$.
Let $\{\widetilde\g(i)\}_{i\in\BZ}$ be the
$\BZ$-grading determined by  $h/2$, i.e., 
$\widetilde{\g}( i)=\g( 2i)$. Then $\dim\g^{\edva}=
\dim\widetilde{\g}( 0)+\dim \widetilde{\g}( 1)=\dim\g( 0)+\dim \g( 2)$
by virtue of Proposition~\ref{stab}(v).   Since the dimension of
all centralisers has the same parity, $\dim\g^e_{nil}$ is even.
\end{proof}

One may refine these necessary conditions using $\BN$-gradings of
centralisers. We assume that each centraliser is equipped with the "natural" $\BN$-grading, i.e., those determined by its own  characteristic.
 
\begin{prop}  
If $\ede$ is divisible, then
\begin{itemize}
\item[(a)] \ $\dim\g^e( 2i)+\dim\g^e( 2i{+}2)=
\dim\g^{\edva}( i)$ for all $i\ge 0$ \ and 
\item[(b)] \ 
$\dim\g^e( 4j{-}2)+\dim\g^e( 4j)$ is even for 
all $j\ge 1$.
\end{itemize}
\end{prop}\begin{proof}
%Keep the notation of the previous proof. % of Proposition~\ref{first}. 
For $i\ge 0$, there are the surjective mappings:
\begin{gather*}
 \g( 2i)\overset{\ad e}{\longrightarrow}\dim\g( 2i{+}2)
\overset{\ad e}{\longrightarrow}\dim\g( 2i{+}4), \\
 \widetilde{\g}( i)=
\g( 2i)\overset{\ad \edva}{\longrightarrow}\dim\g( 2i{+}4) =\widetilde{\g}( i{+}2).
\end{gather*} 
This yields (a).
% $\dim\g^e( 2i)+\dim\g^e( 2i{+}2)=\dim\g^{\edva}( i)$. 
(Recall that the grading of $\g^{\edva}$ is
determined by $h/2$.) 
It is well known that, for any $x\in\N$,
$\dim\g^x(i)$ is even whenever $i$ is odd. (E.g. this readily follows from
\cite[Prop.\,1.2]{aif99}.)
Applying this to $x=\edva$ with  $i=2j-1$, we get (b).
\end{proof}

Since $\g^e( 2)\ne 0$, applying the proposition with $i=1$
shows that $\g^{\edva}( 1)\ne 0$.

\begin{rmk}   \label{rem:ind-D}
In \cite[\S\,2]{dy}, Dynkin defined the {\it index\/}  of a simple subalgebra of a simple Lie 
algebra, which is always a nonnegative integer. 
%In particular, one can consider the index of the subalgebra generated by an $\tri$-triple. 
Let $\ind(e)$ denote the index of the subalgebra generated by $\{e,h,f\}$.
It is easily seen that if $\ede$ is divisible, then
$\ind (e)=4\, \ind(\edva)$, i.e., $\ind(e)/4\in \BN$. (The proof essentially boils down to the
equality $\eus K(h,h)/\eus K(h/2,h/2)=4$.) 
It is worth noting that $\ind(e)$ can be odd  for an even nilpotent element $e$. 
Hence the condition that 
$\ind(e)/4\in \BN$ is not vacuous. 
\end{rmk}
\begin{rmk}   \label{es}
Let $S$ be a connected semisimple subgroup of $G$ with Lie algebra $\es$.  Clearly,
if $e\in\N\cap\es$ and the orbit $S{\cdot}e$ is divisible, then so is
$G{\cdot}e$. 
%For,  $\tri$-triples $\{e,h,f\}$ and  $\{\edva,h/2, f^{\langle 2\rangle}\}$ lying inside $\es$ are  
%also suitable for $\g$.
%%This is trivial, but not useless, because the weighted
%%Dynkin diagrams of $S{\cdot}e$ and $G{\cdot}e$ may look quite different.
But the converse is not always true.
The simplest (counter)example is guaranteed by Morozov and Jacobson:
any nonzero nilpotent element is included in $\tri$, but the nilpotent orbit in
$\tri$ is not divisible.
%More interesting examples can be found with the help of tables in \cite{dy} and  \cite{alela}.
\end{rmk}

%%%%%%%%%%%%%%%%%%%%%%%%%%%%%
\section{Classification of divisible orbits}
%%%%%%%%%%%%%%%%%%%%%%%%%%%%%

\subsection{The classical cases}
Let $\BV$ be a finite-dimensional $\bbk$-vector space and $\g=\g(\BV)$  a classical simple 
Lie algebra, i.e.,  $\slv$, or $\sov$, or $\spv$. In the last two cases, $\BV$ is endowed with
a bilinear non-degenerate form $\Phi$, which is symmetric or skew-symmetric, respectively.
It is customary to represent the nilpotent orbits (elements) by
partitions of $\dim\BV$, and our criterion for $\ede$ to be divisible is given
in terms of  partitions. 

Recall that $\boldsymbol{\lb}=(\lb_1,\dots,\lb_n)$ is a partition of $N$ if
$\sum\lb_i=N$ and $\lb_1\ge\lb_2\ge\dots\ge\lb_n > 0$.
For $e\in\N\subset\g(\BV)$, let $\boldsymbol{\lb}[e]$ denote the corresponding partition
of $N=\dim\BV$. If 
$\boldsymbol{\lb}[e]=(\lb_1,\dots,\lb_n)$, then we can decompose $\BV$ into a
sum of cyclic $e$-modules (Jordan blocks):  
\beq  \label{eq:razlozh}
    \BV=\bigoplus_{i=1}^n \BV[i] ,
\eeq 
where $\dim\BV[i]=\lb_i$. For all classical Lie algebras, the explicit formulae for $\htt(e)$ in terms of 
$\boldsymbol{\lb}[e]$ are given in \cite{aif99}. We recall them below.

\begin{thm}  \label{sl+sp+so}
Let $e\in \g(\BV)$ be a nilpotent element with  partition $\boldsymbol{\lb}[e]=(\lb_1,\dots,
\lb_n)$.
\begin{itemize}
\item[\sf (i)]\ Suppose $\g=\slv$ or $\spv$. Then $\ede$ is divisible if and only if
all $\lb_i$ are odd.
\item[\sf (ii)]\  Suppose $\g=\sov$. Then $\ede$ is divisible if and only if
the following conditions hold:
\begin{itemize}
\item all $\lb_i$ are odd.
\item if $\lb_{2k+1}=4l+3$, then $\lb_{2k+2}=4l+3$ as well;
\item if $\lb_{2k+1}=4l+1>1$, then $\lb_{2k+2}=4l+1$ or $4l-1$. \\(There is no further conditions if $\lb_{2k+1}=1$.)
\end{itemize}
\end{itemize}
In all  cases, $\boldsymbol{\lb}[\edva]$ is obtained
by the following procedure: each odd part $\lb_i=2l+1\ge 3$ is replaced with two
parts $l+1$ and $l$. The resulting collection of parts determines the
required partition.
\end{thm}\begin{proof}
For all classical Lie algebras, $e$ is even if and only
if all the parts of $\boldsymbol{\lb}[e]$ have the same parity.
%For  $\ede$ divisible,  it follows from \eqref{height}
%that $\htt(\edva)=\frac{1}{2}\htt(e)$.

1) The proof for $\slv$ is quite simple. By \cite[Theorem\,2.3]{aif99},
$\htt(e)=2(\lb_1-1)$; in particular, the height of any nilpotent element is even.
If $\ede$ is divisible, then $e$ is even and
all parts of $\boldsymbol{\lb}[e]$ have the same parity. Since
$\htt(\edva)=\frac{1}{2}\htt(e)$ should also be even,
$\lb_1$ must be odd.

Conversely, if all $\lb_i$'s are odd, then we set $\edva :=e^2$,
the usual matrix power. 
First, it is easily seen that  $[h/2,e^2]=2e^2$ whenever $[h,e]=2e$; %for any $\tri$-triple;
second,  one readily verifies that 
$h/2\in\Ima(\ad(e^2))$ if and only if all $\lb$'s are odd.
(This can be done for each Jordan block $\BV[i]$ separately.)
Thus, $h/2$ is a characteristic of $e^2$.
Finally, under the passage $e\mapsto e^2$,
every Jordan block of size $2k+1$ is replaced
with two blocks of size $k+1$ and $k$.

2) For $\g=\spv$, the partitions $\boldsymbol{\lb}[e]$ are characterized by the property
that each part of odd size occurs an even number of times. 
Since $\htt(e)$ is given by the same formula as in 1), the necessity is obtained 
analogously.

Conversely, suppose that all $\lb_i$'s are odd. We cannot merely take 
$\edva=e^2$, since $e^2\not\in\spv$. However, the procedure can slightly be adjusted. 
In our setting, each part of $\boldsymbol{\lb}[e]$
occurs an even number of times and hence $\dim\BV[2i{-}1]=\dim\BV[2i]$.
Since $\dim \BV[i]$ is odd for all $i$, the skew-symmetric form $\Phi$ vanishes
on every  $\BV[i]$. 
%Let $\BV[i]$ be the space of $i$-th Jordan block, so that $\BV=\oplus_{i=1}^n\BV[i]$ 
%and $\dim \BV[i]=\lb_i$. One can choose the Jordan normal form of $e$
However, one can arrange the decomposition \eqref{eq:razlozh}
such that, for each pair of indices $(2i{-}1,2i)$, 
%we have $\dim\BV[i_1]=\dim\BV[i_2]$ and the 
$\Phi$ is non-degenerate on $\BV[2i{-}1]\oplus\BV[2i]$. Then
it suffices to define $\edva$ separately on each sum of this form.
That is, without loss of generality, we may assume that 
$\boldsymbol{\lb}[e]=(2l+1,2l+1)$  and $\BV=\BV[1]\oplus\BV[2]$.
Now, define $\edva$ as follows:
\[
    \edva\vert_{\BV[1]}=e^2\ \text{ and }\ \edva\vert_{\BV[2]}=-e^2\ .
\]
A straightforward verification shows that $\edva\in\spv$ and $h/2$ is a 
characteristic of $\edva$.

3)  For $\g=\sov$, the partitions $\boldsymbol{\lb}[e]$ are characterized by the property
that each part of even size occurs an even number of times. 
Here we have \cite[Theorem\,2.3]{aif99}:
\[
\htt(e)=\left\{ \begin{array}{ll} 
\lb_1+\lb_2-2, & \text{ if\quad } \lb_2\ge \lb_1-1 \\ 
2\lb_1-4, & \text{ if\quad } \lb_2\le \lb_1-2\, . \end{array} \right. 
\]
In particular, either $\htt(e)$ is even or $\htt(e)\equiv\! 3\pmod 4$.

 \textbullet \quad Suppose that $\ede$ is divisible.
If $\lb_1$ is even, then all parts are even, i.e.,
$\boldsymbol{\lb}[e]$ is a {\it very even\/}
partition, and we are in the type $\GR{D}{}$ case.
Associated to a very even partition, one has two nilpotent orbits
whose weighted Dynkin diagrams differ only at the "very end".
Namely, 
$\eus D_1{=}\ast\cdots \ast\,\text{\begin{tabular}{@{}c@{}}
\lower3.5ex\vbox{\hbox{{2}\rule{0ex}{2.5ex}}
\hbox{\hspace{0.4ex}\rule{0ex}{1ex}\rule{0ex}{1.4ex}}\hbox{{0}\strut}}
\end{tabular}}$ and 
$\eus D_2{=}\ast\cdots\ast$\,\begin{tabular}{@{}c@{}}
\lower3.5ex\vbox{\hbox{{0}\rule{0ex}{2.5ex}}
\hbox{\hspace{0.4ex}\rule{0ex}{1ex}\rule{0ex}{1.4ex}}\hbox{{2}\strut}}
\end{tabular}\, . 
If such a $\eus D_i$ were divisible, then 
$\ast\cdots \ast$\,\begin{tabular}{@{}c@{}}
\lower3.5ex\vbox{\hbox{{1}\rule{0ex}{2.5ex}}
\hbox{\hspace{0.4ex}\rule{0ex}{1ex}\rule{0ex}{1.4ex}}\hbox{{0}\strut}}
\end{tabular} \ 
would be a weighted Dynkin diagram, too. But this is impossible,
because the sum of two last marks is always even in the $\GR{D}{}$-case~\cite[IV-2.32]{ss}.
Hence all $\lb_i$ must be odd.

For  $\lb_i=2m_i+1$, the $h$-eigenvalues on $\BV[i]$
are $\{2m_i,2m_i{-}2,\ldots,-2m_i\}$ and hence the $(h/2)$-eigenvalues   
are $\{m_i,m_i{-}1,\ldots,-m_i\}$. If $h/2$ is again the semisimple element of an 
$\tri$-triple, then the resulting set of eigenvalues on $\BV$ corresponds to the Jordan 
normal form, where each block of size $2m_i{+}1$ is replaced with
two blocks of sizes $m_i{+}1$ and $m_i$. (The structure of 
$\BV$ as  $\tri$-module is fully determined by the eigenvalues of the 
semisimple element.) Hence,  the Jordan 
normal form of $\edva$ is uniquely determined by that  of $e$.
However, the resulting partition must be "orthogonal", which  leads 
precisely to the remaining conditions in (ii). Indeed, suppose that
$\lb_1=4m+3$. This yields parts $(2m+2,2m+1)$ in 
$\boldsymbol\lb[\edva]$. Since part $2m+2$ should occur an even number of times, we must
have $\lb_2=4m{+}3$. For $\lb_1=4m+1$ with $m>0$, we obtain parts 
$(2m+1,2m)$ in  $\boldsymbol\lb[\edva]$. 
Since  part $2m$ must occur an even number of times, we must
have $\lb_2\in\{4m{+}1, 4m{-}1\}$. Then splitting away the subspace 
$\BV[1]\oplus\BV[2]$, we argue by induction.

\textbullet \quad Conversely, suppose that $\boldsymbol{\lb}[e]$ satisfies all  
conditions in (ii).  Then the total number of parts that are greater than 1 is even.
If, say, $\lb_{2k}>1$ and $\lb_{2k+1}=1$, then we split $\BV$ into the direct sum of
spaces $\BV_j$, $j=1,\dots,k+1$, where $\BV_j:=\BV[2j{-}1]\oplus\BV[2j]$ for $j\le k$ and
$\BV_{k+1}$ is the sum all Jordan blocks of size 1. In other words, 
$\BV_{k+1}\subset \BV$ is the fixed-point subspace of the algebra $\langle e,h,f\rangle$.
Without loss of generality, we may assume that 
%Splitting $\BV$ into the direct sum  of spaces $\BV_k:=\BV[2k{-}1]\oplus\BV[2k]$, where 
$\Phi\vert_{\BV_j}$ is non-degenerate for all $j$.
We shouldn't do anything with $\BV_{k+1}$, and all other $\BV_j$ can be treated
separately.  Therefore, we may assume that $k=1$.
%we may assume that $\boldsymbol{\lb}[e]$ has only two parts.
Now, there are two possibilities.

(a) If $\lb_1=\lb_2$, then we can argue as for $\spv$.  Since $\dim\BV[1]=\dim\BV[2]$,
it can be arranged that both $\BV[1]$ and $\BV[2]$ are isotropic with respect to 
$\Phi$. Then we set
\[
    \edva\vert_{\BV[1]}=e^2\ \text{ and }\ \edva\vert_{\BV[2]}=-e^2\ .
\]
A straightforward verification shows that $\edva\in\sov$ and $h/2$ is a 
characteristic of $\edva$.

(b) Assume that $\lb_1=4m+1$ and $\lb_2=4m-1$. This is the most interesting 
case, because now $\edva$ will not preserve the Jordan blocks of $e$.
Here $\Phi$ is non-degenerate on both $\BV[1]$ and $\BV[2]$.
Let $\{v_i\mid i=1,\ldots,4m+1\}$ be a basis for $\BV[1]$ and
$\{w_i\mid i=2,\ldots,4m\}$  a basis for $\BV[2]$.
Without loss of generality, one may assume that 
$e(v_i)=v_{i+1}$, $e(w_i)=w_{i+1}$, $\Phi(v_i,v_{4m+2-j})=(-1)^{i-1}\delta_{i,j}$,
and  $\Phi(w_j,w_{4m-i})=(-1)^j\delta_{i,j}$. 
%%Here $(\ ,\ )$ is a symmetric bilinear form on $\BV$
%%preserved by $\sov$ and all other scalar products of basis vectors are equal to zero.
Define $\edva\in\glv$ as follows:
\[  \edva: \left\{ \begin{array}{c}
v_1\mapsto -w_3\mapsto v_5 \mapsto \ldots \mapsto -w_{4m-1}\mapsto v_{4m+1}\mapsto 0 \\
v_2\mapsto -w_4\mapsto v_6 \mapsto \ldots \mapsto -w_{4m}  \mapsto 0 \\
w_2\mapsto -v_4\mapsto w_6 \mapsto \ldots \mapsto -v_{4m}  \mapsto 0 \\
v_3\mapsto -w_5\mapsto v_7 \mapsto \ldots \mapsto -w_{4m-3}\mapsto v_{4m-1}\mapsto 0 
\end{array}\right.
\]
Then $\boldsymbol\lb[\edva]=(2m+1,2m,2m,2m-1)$.
It is not hard to check that $\edva\in \sov$ and $h/2$ is a characteristic of $\edva$.
\end{proof}%
{\bf Warning.}  For a divisible $e\in\slv$, one can take $\edva=e^2$. However, 
this procedure may not simultaneously apply to $f$.
Given $e^2$ and $h/2$, the last element
of the $\tri$-triple is uniquely determined, but it is not necessarily a multiple of 
$f^2$. It is instructive to 
consider a regular nilpotent $e\in\mathfrak{sl}_5$.

\subsection{The exceptional cases}
If $\g$ is exceptional, then one can merely browse %look through 
the list of the weighted Dynkin diagrams and pick the suitable pairs among them.
The output is presented  below. We use the standard notation for nilpotent
orbits in the exceptional Lie algebras that goes back to Dynkin and Bala--Carter
(see e.g. \cite[Ch.\,8]{CM}). 
The meaning of the first and last columns is explained in Section~\ref{sect:A2}.

\begin{longtable}{c|ccccc|}
\caption{The friendly pairs and divisible Dynkin diagrams in the exceptional algebras} 
\label{table_E}
%\begin{tabular}
\endfirsthead
\multicolumn{4}{l}{The friendly pairs, cont. } \\ \hline
& reachable & $G{\cdot}\edva$ & $G{\cdot}e$ & $\ede$ & $\GR{A}{2}$-pair  \\ \hline
\endhead
\endfoot    
 & reachable & $G{\cdot}\edva$ & $G{\cdot}e$ & $\ede$ & $\GR{A}{2}$-pair \\ \hline
\framebox{$\GR{E}{6}$} & + &
$\GR{A}{1}$  & $\GR{A}{2}$ & \begin{E6}{0}{0}{0}{0}{0}{2}
\end{E6} &{\bf\color{green} +} \\
& + & $2\GR{A}{1}$  & $2\GR{A}{2}$ & \begin{E6}{2}{0}{0}{0}{2}{0}
\end{E6} & {\bf\color{green} +} \\
& + & $3\GR{A}{1}$  & $\GR{D}{4}(a_1)$ & \begin{E6}{0}{0}{2}{0}{0}{0}
\end{E6}  & {\bf\color{green} +} \\
& + & $\GR{A}{2}+\GR{A}{1}$  & $\GR{A}{4}$  & \begin{E6}{2}{0}{0}{0}{2}{2}
\end{E6} & {\bf\color{green} +} \\
& + & $2\GR{A}{2}{+}\GR{A}{1}$  & $\GR{E}{6}(a_3)$ & \begin{E6}{2}{0}{2}{0}{2}{0}
\end{E6} & {\bf\color{green} +} \\
& -- & $\GR{A}{4}+\GR{A}{1}$  & $\GR{E}{6}(a_1)$  & \begin{E6}{2}{2}{0}{2}{2}{2}
\end{E6}  & {\bf\color{green} --}\\ \hline
\framebox{$\GR{E}{7}$} 
& + & $\GR{A}{1}$  & $\GR{A}{2}$ & \begin{E7}{0}{0}{0}{0}{0}{2}{0}\end{E7} & {\bf\color{green} +}  \\
& + &$2\GR{A}{1}$ & $2\GR{A}{2}$ & \begin{E7}{0}{2}{0}{0}{0}{0}{0}\end{E7} & 
{\bf\color{green} +}  \\
& + &$(3\GR{A}{1})'$  & $\GR{D}{4}(a_1)$ & \begin{E7}{0}{0}{0}{0}{2}{0}{0}\end{E7} & 
{\bf\color{green} +}  \\
& + &$\GR{A}{2}{+}\GR{A}{1}$  & $\GR{A}{4}$ & \begin{E7}{0}{2}{0}{0}{0}{2}{0}\end{E7} & 
{\bf\color{green} +} \\
& + & $\GR{A}{2}{+}2\GR{A}{1}$  & $\GR{A}{4}{+}\GR{A}{2}$ & \begin{E7}{0}{0}{0}{2}{0}{0}{0}
\end{E7} & {\bf\color{green} +}  \\
& + & $2\GR{A}{2}{+}\GR{A}{1}$ & $\GR{E}{6}(a_3)$ & 
\begin{E7}{0}{2}{0}{0}{2}{0}{0}\end{E7} & {\bf\color{green} +} \\
& -- & $\GR{A}{3}{+}\GR{A}{2}$  & $\GR{A}{6}$ & \begin{E7}{0}{2}{0}{2}{0}{0}{0}\end{E7} & 
{\bf\color{green} --} \\
& + &$\GR{A}{4}{+}\GR{A}{1}$  & $\GR{E}{6}(a_1)$ & \begin{E7}{0}{2}{0}{2}{0}{2}{0}\end{E7}  & 
{\bf\color{green} +}
\\ \hline
\framebox{$\GR{E}{8}$}
& +&  $ \GR{A}{1}$ & $ \GR{A}{2}$ & \begin{E8}{2}{0}{0}{0}{0}{0}{0}{0}\end{E8} & {\bf\color{green} +}\\
& +& $2\GR{A}{1}$ & $2\GR{A}{2}$ & \begin{E8}{0}{0}{0}{0}{0}{0}{2}{0}\end{E8} & {\bf\color{green} +} \\
& +& $3\GR{A}{1}$ & $\GR{D}{4}(a_1)$ & \begin{E8}{0}{2}{0}{0}{0}{0}{0}{0}\end{E8} &{\bf\color{green} +} \\
& + & $4\GR{A}{1}$ & $\GR{D}{4}(a_1){+}\GR{A}{2}$ & 
\begin{E8}{0}{0}{0}{0}{0}{0}{0}{2}\end{E8} & {\bf\color{green} +}\\
& + &$\GR{A}{2}{+}\GR{A}{1}$  & $\GR{A}{4}$ & \begin{E8}{2}{0}{0}{0}{0}{0}{2}{0}\end{E8} &{\bf\color{green} +} \\
& + &$\GR{A}{2}{+}2\GR{A}{1}$ & $\GR{A}{4}{+}\GR{A}{2}$ & 
\begin{E8}{0}{0}{2}{0}{0}{0}{0}{0}\end{E8} & {\bf\color{green} +}\\
& + &$2\GR{A}{2}{+}\GR{A}{1}$  & $\GR{E}{6}(a_3)$ & 
\begin{E8}{0}{2}{0}{0}{0}{0}{2}{0}\end{E8} & {\bf\color{green} +} \\
& + & $2\GR{A}{2}{+}2\GR{A}{1}$  & $\GR{E}{8}(a_7)$ & 
\begin{E8}{0}{0}{0}{2}{0}{0}{0}{0}\end{E8} & {\bf\color{green} +}\\  
&  -- & $\GR{A}{3}{+}\GR{A}{2}$  & $\GR{A}{6}$ & \begin{E8}{0}{0}{2}{0}{0}{0}{2}{0}\end{E8} & 
{\bf\color{green} --}\\
&  + & $\GR{A}{4}{+}\GR{A}{1}$  & $\GR{E}{6}(a_1)$ & 
\begin{E8}{2}{0}{2}{0}{0}{0}{2}{0}\end{E8} &{\bf\color{green} +} \\

& + & $\GR{A}{4}{+}2\GR{A}{1}$  & $\GR{E}{8}(b_6)$ & 
\begin{E8}{2}{0}{0}{0}{2}{0}{0}{0}\end{E8} & {\bf\color{green} --}\\

& + & $\GR{A}{4}{+}\GR{A}{3}$  & $\GR{E}{8}(a_6)$ & 
\begin{E8}{0}{2}{0}{0}{2}{0}{0}{0}\end{E8} & {\bf\color{green}--}\\

& -- & $\GR{D}{7}(a_2)$  & $\GR{E}{8}(a_4)$ & \begin{E8}{2}{0}{2}{0}{2}{0}{2}{0}\end{E8} & 
{\bf\color{green} --}\\
\hline
\framebox{$\GR{F}{4}$} &+ &$\GR{A}{1}$  &  \rule{0pt}{13pt} $\GR{A}{2}$ & 0--0$\Leftarrow$0--2 
& {\bf\color{green} +}\\
& + & $\GRt{A}{1}$  & \rule{0pt}{13pt} $\GRt{A}{2}$ & 2--0$\Leftarrow$0--0 & {\bf\color{green} +} \\
& + &$\GR{A}{1}{+}\GRt{A}{1}$  &  \rule{0pt}{13pt} $\GR{F}{4}(a_3)$ & 0--0$\Leftarrow$2--0  
& {\bf\color{green} +}\\
& --&$\GR{A}{1}{+}\GRt{A}{2}$  &  \rule{0pt}{13pt} $\GR{F}{4}(a_2)$ & 2--0$\Leftarrow$2--0 
&  {\bf\color{green} --}\\ \hline
\framebox{$\GR{G}{2}$}\rule{0pt}{13pt} & +
& $\GR{A}{1}$  & $\GR{G}{2}(a_1)$ & 0$\Lleftarrow$2 &{\bf\color{green} +} \\ 
\hline
\end{longtable}

\noindent
Recall that, for every orbit $\co=G{\cdot}e\subset\N$, any two minimal Levi subalgebras meeting 
$G{\cdot}e$  are $G$-conjugate \cite[Theorem\,8.1.1]{CM}. 
If $\el$ is such a minimal Levi subalgebra and $e\in\el$, then
the notation of Table~\ref{table_E}
represents the Cartan type of $\el$,  with some 
additional data (like $(a_i)$ or $(b_i)$) if the  orbit $L{\cdot}e$ in $\el$ is not regular. 
(See \cite[8.4]{CM} for the details.)  
If $\g$ itself is the minimal Levi subalgebra
meeting $\co$, then $\co$ is called {\it distinguished}. This is equivalent to that
$\g^e_{red}=\{0\}$ for $e\in\co$.
For instance,  the third row for $\GR{E}{6}$ contains the divisible orbit $G{\cdot}e$
denoted by
$\GR{D}{4}(a_1)$. This means that a minimal Levi subalgebra, $\el$,  meeting $G{\cdot}e$
is of type $\GR{D}{4}$ and the intersection $\el\cap G{\cdot}e$ is the distinguished
$SO_8$-orbit, which is called $\GR{D}{4}(a_1)$. In fact, it is 
the subregular nilpotent orbit in $\mathfrak{so}_8$, and its partition is $(5,3)$.

\begin{thm}   \label{thm:divis-Levi}
Let $\el$ be a minimal Levi subalgebra of $\g$ containing $e$.
Then $G{\cdot}e$ is divisible if and only if $L{\cdot}e$ is.
\end{thm}
\begin{proof}
We have only to prove that if $G{\cdot}e$ is divisible, then so is $L{\cdot}e$.
In other words, if $G{\cdot}e$ is divisible, then $G{\cdot}\edva\cap \el\ne\varnothing$.
Our case-by-case proof is based on the previous classification. 
I hope there is a better proof.

1. $\g=\slv$. If $\boldsymbol\lb[e]=(\lb_1,\dots,\lb_n)$, then $[\el,\el]$
%a minimal Levi subalgebra meeting $G{\cdot}e$ 
is of type $\GR{A}{\lb_1-1}+\dots+\GR{A}{\lb_n-1}$, and
the component of $e$ in each summand is
a regular nilpotent element there.
By Theorem~\ref{sl+sp+so}(i), the regular nilpotent orbit in $\GR{A}{m}$ is divisible
if and only if $m$ is even.

2. $\g=\spv$. If $G{\cdot}e$ is divisible, then $\boldsymbol\lb[e]=(\nu_1^{2k_1},\dots,
\nu_m^{2k_m})$, where $\nu_1> \dots > \nu_m>0$ and all $\nu_i$ are odd.
Here $[\el,\el]$ is of type 
$k_1\GR{A}{\nu_1-1}+\dots+k_m\GR{A}{\nu_m-1}$, and the rest is the same as in part 1.

3. $\g=\sov$. Recall that  $e\in\sov$ is distinguished if and only if
all parts of $\boldsymbol\lb[e]$ are  different (and hence odd).
Suppose $G{\cdot}e$ is divisible, i.e., $\boldsymbol\lb[e]$ satisfies the conditions of 
Theorem~\ref{sl+sp+so}(ii).
Then $\boldsymbol\lb[e]$ may have repeating odd parts.  
Each pair of equal parts in  $\boldsymbol\lb[e]$ determines a
summand of type $\GR{A}{\lb_i-1}$ in $\el$, and the projection of $e$
to this summand is regular nilpotent.
Discarding all  pairs of equal parts (if any), we get
a partition of the form $(4l_1+1,4l_1-1,\dots,4l_m+1,4l_m-1, (1))$, where 
$l_1>l_2>\ldots >l_m>0$ and the last "1" is optional (it occurs if and only if $\dim\BV$ is odd).
The remaining partition represents a (distinguished) divisible orbit in %a Levi subalgebra 
$\mathfrak{so}(\BV')\subset \sov$. Note that  $\dim\BV-\dim \BV'$ is  even, hence 
$\mathfrak{so}(\BV')$ is the derived algebra of a Levi subalgebra of $\sov$.

4. For $\g$ exceptional,  it suffices to understand information encoded in column
"$G{\cdot}e$" in Table~\ref{table_E}
(see explanations above). For instance, the last divisible orbit for $\g=\GR{E}{7}$
is called $\GR{E}{6}(a_1)$. This means that  $[\el,\el]$ is of type
$\GR{E}{6}$ and the corresponding distinguished $\GR{E}{6}$-orbit  is  %, of class  
$\GR{E}{6}(a_1)$. Now, the last item in the $\GR{E}{6}$-part of the
table shows that this orbit is also divisible.  If $[\el,\el]$ is of classical type, then one should again use Theorem~\ref{sl+sp+so}.
\end{proof}

\begin{rmk}
Since $\dim\g(2)> \dim\g(4)$, we have $\dim\g^e_{red} < \dim \g^{\edva}_{red}$. 
Moreover, $\edva=e^2$ for $\g=\slv$, and therefore $\slv^e\subset \slv^{\edva}$ and
$\slv^e_{red}\subset \slv^{\edva}_{red}$.
This does not mean, however, that  the inclusion $\g^e_{red} \subset \g^{\edva}_{red}$ 
always holds for a suitable choice of $\edva$.
For instance, for the divisible orbit $\GRt{A}{2}$ in $\g=\GR{F}{4}$, one has 
$\g^e_{red}=\GR{G}{2}$ and $ \g^{\edva}_{red}=\GR{A}{3}$. 
Recall that a minimal Levi subalgebra $\el$ meeting $G{\cdot}e$ is obtained as follows:
If $\h$ is a Cartan subalgebra of $\g^e(0)$, then $\el=\z_\g(\h)$ \cite[Ch.\,8]{CM}.
Consequently, 
Theorem~\ref{thm:divis-Levi} is equivalent to the assertion that 
a Cartan subalgebra of $\g^e(0)=\g^e_{red}$ is contained in a Cartan subalgebra of $\g^{\edva}(0)=\g^{\edva}_{red}$.  
This also implies that $\rk(\g^e_{red})<\rk(\g^{\edva}_{red})$.
\end{rmk}

%%%%%%%%%%%%%%%%%%%%%%%%%%%%%%%%%%%%%%
\section{$\GR{A}{2}$-pairs of orbits and reachable elements}  
\label{sect:A2}

\noindent
In this section, an interesting class of friendly pairs is studied. 

\begin{df}   \label{def:pair}
A pair of nilpotent orbits $(\tilde\co,\co)$ is said to be an  $\GR{A}{2}$-{\it pair\/}, 
if there is a simple subalgebra $\sltri\subset\g$  such that $\tilde\co\cap\sltri$  is the 
regular nilpotent orbit  and $\co\cap\sltri$  is the minimal nilpotent orbit in $\sltri$.
Then $\tilde\co$ (resp. $\co$) is called an {\it upper\/}-$\GR{A}{2}$ 
(resp. {\it low\/}-$\GR{A}{2}$) {\it orbit}. 
\end{df}

\noindent
The property of being an $\GR{A}{2}$-pair imposes strong constraints
on both orbits, so that there are only a few $\GR{A}{2}$-pairs in simple Lie algebras.

%As in \cite{reach}, 
We say that $e\in\N$ (or the orbit $G{\cdot}e$) is {\it reachable\/}, if 
$e\in [\g^e, \g^e]$. This property was first considered in 
\cite{EG}, where such nilpotent elements are called "compact". 
Some further results are obtained in \cite{reach}.

\begin{lm}
Let $(\tilde\co, \co)$ be an $\GR{A}{2}$-pair. Then it is a friendly pair
(i.e., $\tilde\co$ is divisible and $\co=\tilde\co^{\langle 2\rangle}$)  and $\co$ is reachable.
\end{lm}
\begin{proof}
The required properties obviously hold for two orbits in $\sltri$.
This implies the assertion for orbits in $\g$.
\end{proof}%

\noindent
Reachable nilpotent elements (orbits) have some intriguing properties that are not fully understood yet.
For instance, explicit classification shows that $\co\subset \N$ is reachable if and
only if $\codim_{\ov{\co}} (\ov{\co}\setminus\co)\ge 4$ \cite{EG}. 
It is a challenging task to find an {\sl a priori\/} 
relationship between two such different properties. 
In \cite[4.6]{reach}, we posed the following question:
\\[.7ex]  \centerline{$(\diamondsuit)$ \hfil
{\sl Is it true that if $e\in\N$ is reachable, then $\g^e_{nil}$ 
is generated as Lie algebra by  $\g^e(1)$?}}
\\[.7ex]
This was proved for $\g=\slv$ \cite[Theorem\,4.5]{reach}.
Below,  we prove a stronger assertion for reachable orbits occurring as 
low-$\GR{A}{2}$ orbits (Theorem~\ref{thm:reach}). 
To this end, we need some notation and results on 
$\sltri$.

Fix  a triangular decomposition $\ut_-\oplus\te\oplus\ut=\sltri$. 
Let $\ap_1,\ap_2,\ap_1+\ap_2=\theta$ be the positive roots of $\sltri$ and 
$e_1,e_2, e=e_\theta$ the corresponding root vectors in $\ut$.
Then $\tilde e=e_1+e_2$ is  a regular nilpotent element and $\etdva=e$.
Let $h\in \te$ be the characteristic of $\tilde e$  (and hence $h/2$ is a characteristic of $e$).
Let $U$ be the maximal unipotent  subgroup of $SL_3$ corresponding to 
$\ut$ and $U_\theta$ the root  subgroup corresponding to $\theta$. Then $U/U_\theta
\simeq (\mathbb G_a)^2$ is commutative.

Let $\varpi_i$ be the fundamental weight coresponding to $\ap_i$.
The simple  $SL_3$-module with highest weight $a\varpi_1+
b\varpi_2$ ($a,b\ge 0$) is denoted by $\mathsf R(a,b)$. 
%By Weyl's formula, $\dim\mathsf R(a,b)=(a+1)(b+1)(a+b+2)/2$.
Let $\esi_1,\esi_2,\esi_3$ be the  $T$-weights of $\mathsf R(1,0)$ such that 
$\ap_1=\esi_1-\esi_2$ and $\ap_2=\esi_2-\esi_3$.

As is well known,  $SL_3/U$ is quasi-affine and 
$\bbk[SL_3/U]$ is a
{\it model algebra}, i.e., each finite-dimensional simple $SL_3$-module occurs
exactly once in it.  Set $X:=\spe(\bbk[SL_3/U])$. It is an affine $SL_3$-variety containing  
$SL_3/U$ as a dense open subset.
One can explicitly realise $X$   
as a subvariety in $\mathsf R(1,0)\oplus \mathsf R(0,1)$, 
the sum of the fundamental representations. 
(This is also true for  an arbitrary semisimple $G$ in place of $SL_3$ \cite{VP}.) 
Since $\dim X=5$, it is a hypersurface in $\mathsf R(1,0)\oplus \mathsf R(0,1)$.
Let $\ah$ be the simple three-dimensional subalgebra of $\sltri$ containing $e$
and $h/2$.

\begin{thm}   \label{thm:ab}
%\leavevmode\par 
{\sf (i)} \ $\bbk[SL_3/U]^{U_\theta}$ is a  polynomial algebra 
of Krull dimension $4$ whose free generators can be explicitly described.

{\sf (ii)} \  For any $(a,b)\in \BN^2$, $\mathsf R(a,b)^{U_\theta}$ is a cyclic  $U/U_\theta$-module of dimension $(a+1)(b+1)$.  More precisely, there is a unique (up to a multiple)
cyclic vector that is a $T$-eigenvector.

{\sf (iii)} \  The branching rule $\sltri\downarrow\ah$ is given by the formula (for $a\ge b$)
\[
    \mathsf R(a,b)\vert_\ah=\mathsf R_0\oplus 2\mathsf R_1\oplus \dots \oplus 
    (b+1)\mathsf R_b\oplus 
    \dots \oplus (b+1)\mathsf R_a\oplus b\mathsf R_{a+1}\oplus \dots \oplus \mathsf R_{a+b},
\]
where $\mathsf R_n$ is the simple $\ah$-module of dimension $n+1$. The cyclic 
vector from (ii) lies in the unique 1-dimensional submodule 
$\mathsf R_0\subset \mathsf R(a,b)$.
\end{thm}

\begin{rema} Parts (i) and (ii) are particular instances of a general assertion, which
is valid for all semisimple $G$ in place of $SL_3$ and  the derived group 
$(U,U)$ in place of $U_\theta$ \cite[Theorems 1.6, 1.8]{odno-sv}.
For reader's convenience, we give a self-contained proof in the $SL_3$-case.
\end{rema}
\begin{proof}  (i) \ Choose the functions $x_1,x_2,x_3$ (resp. $\xi_1,\xi_2,\xi_3$) such that 
they form a $T$-weight basis for $\mathsf R(1,0)\subset \bbk[X]$ (resp. 
$\mathsf R(0,1)\subset \bbk[X]$). Assume  that the weight of $x_i$ is
$\esi_i$ and the weight  of $\xi_i$ is $-\esi_i$.  Then 
$x_1,\dots,\xi_3$ generate $\bbk[X]=\bbk[SL_3/U]$ 
modulo a  relation of the form $\sum_1^3 a_ix_i\xi_i=0$, where $a_i\in \bbk^\times$.
It  follows from the previous description that
 $x_1,x_2,\xi_2,\xi_3$ are $U_\theta$-in\-va\-ri\-ant. 
Thus,
\[
     \bbk[x_1,x_2,\xi_2,\xi_3]\subset \bbk[X]^{U_\theta}
\]  
and both algebras have Krull dimension $4=\dim X-\dim U_\theta$. As the left-hand side algebra 
is algebraically closed in $\bbk[X]$,  they must be equal. 

(ii) \ The vector space decomposition 
$\bbk[X]=\underset{(a,b)\in\BN^2}{\bigoplus}\mathsf R(a,b)$  is  actually a bi-grading, 
and it induces the bi-grading  
\[
  \bbk[X]^{U_\theta}=\underset{(a,b)\in\BN^2}{\bigoplus}\mathsf R(a,b)^{U_\theta} .
\]
Since $x_1,x_2\in \mathsf R(1,0)$ and $\xi_2,\xi_3\in \mathsf R(0,1)$ are free generators
of  $\bbk[X]^{U_\theta}$,
the monomials $\{m(i,j):=x_1^i x_2^{a-i}\xi_2^{b-j} \xi_3^{j}\mid 0\le i \le a, 0\le j\le b\}$
form a basis for $\mathsf R(a,b)^{U_\theta}$. It is convenient to think of this set 
of monomials as a rectangular array of shape $(a+1)\times (b+1)$.

\noindent
The root vectors $e_1,e_2$ form a basis for $\Lie(U/U_\theta)$. Their action
on generators of $\bbk[X]^{U_\theta}$ is given by 
\begin{gather*}
    e_1(x_2)=x_1,  \ e_1(x_1)=0;  \ e_1(\xi_2)=0,  \ e_1(\xi_3)=0, \\
    e_2(\xi_2)=\xi_3,  \ e_2(\xi_3)=0;  \ e_2(x_1)=0,  \ e_2(x_2)=0.
\end{gather*}
Hence $e_1$ (resp. $e_2$) acts along the columns (resp. rows) of that array. Namely,
\\[.6ex]  \indent
$e_1{\cdot} m(i,j)=\begin{cases} m(i+1,j), &  i< a, \\ 0, & i=a. \end{cases}$ \ ; \quad
$e_2{\cdot} m(i,j)=\begin{cases} m(i,j+1), &  j< b, \\ 0, & j=b. \end{cases}$ \ . 
 
 \noindent
Thus, the $T$-eigenvector $m(0,0)=x_2^{a}\xi_2^b$ is the cyclic vector in the 
$U/U_\theta$-module $\mathsf R(a,b)^{U_\theta}$.

(iii) The monomials $x_1^i x_2^{a-i}\xi_2^{b-j} \xi_3^{j}$ are  the highest weight
vectors of all simple $\ah$-modules in $\mathsf R(a,b)$. We have 
$[h/2, x_2]=[h/2,\xi_2]=0$, $[h/2,x_1]=x_1$, and $[h/2,\xi_3]=\xi_3$. Consequently,
%the $h/2$-weights are constant along the diagonals of the above array and
the $k$-eigenspace of $h/2$ is the span of monomials 
$x_1^i x_2^{a-i}\xi_2^{b-j}\xi_3^{j}$ with $i+j=k$.  Counting the number of such monomials
yields the coefficient of $\mathsf R_k$, $0\le k\le a+b$,  in the branching rule.
We also see that the cyclic vector $x_2^{a}\xi_2^b$ is the only $\ah$-invariant
in $\mathsf R(a,b)$.
\end{proof}

\begin{rmk} Here is another way to prove that $\ca:=\bbk[SL_3/U]^{U^\theta}$ is a 
polynomial algebra. Since both $U$ and  $U^\theta$ are unipotent, the algebra $\ca$ 
is  factorial. Let $T\subset SL_3$ be a maximal torus normalising $U$. 
Clearly, $\ca$ admits an effective action of $T\times T$ (via left and right translations).
As  $\spe(\ca)$ is four-dimensional, it is a factorial affine toric  variety. Therefore it is
an affine space. 
\end{rmk}

Now, we return to $\GR{A}{2}$-pairs of orbits in an arbitrary simple algebra $\g$.

\begin{thm}    \label{thm:reach}
Let $\co$ be a low-$\GR{A}{2}$ orbit and $e\in\co$. 
Then there are elements $e_1,e_2\in \g^{e}(1)$ such that
\\[.6ex]
\hbox to \textwidth{ $(\spadesuit)$ \hfil
$\g^{e}(i)=[\g^{e}(i-1),e_1]+[\g^{e}(i-1),e_2]$ \ for each $i\ge 1$.
\hfil}
\\[.7ex]
Consequently, $\g^{e}$ is generated by $\g^{e}(0)$, 
$e_1$, and $e_2$; \ $\g^{e}_{nil}$ is generated by $\g^{e}(1)$; \ 
$\g^{e}_{nil}\subset [\g^e, \g^e]$.
\end{thm}
\begin{proof} 
Take an  $\sltri \subset \g$ such that $\{e,h,f\}\subset\sltri$, $e=e_\theta\in\sltri$ is a highest weight vector
and %$e_1,e_2\in\sltri\cap\g^{e}(1)$ such that
$e_1, e_2$ are simple root elements, as above.
To prove $(\spadesuit)$, we
decompose $\g$ as a sum of simple $\sltri$-modules, 
$\g=\bigoplus_i \mathsf R(a_i,b_i)$.
We have the distinguished submodule $\sltri\simeq \mathsf R(1,1)\subset\g$ with elements
$e_1,e_2\in \g^{e}(1)\cap\sltri$ and $e\in \g^{e}(2)\cap\sltri$.
Since $\g^e=\bigoplus_i \mathsf R(a_i,b_i)^e$,
it suffices to check $(\spadesuit)$ for each  $\mathsf R(a_i,b_i)$ separately. 
That is, we have to prove that
\[
  R(a_i,b_i)^e(i)=[R(a_i,b_i)^e(i-1),e_1]+[R(a_i,b_i)^e(i-1),e_2] \ \text{ for each $i\ge 1$}.
\]
By Theorem~\ref{thm:ab}(ii),(iii), every $\mathsf R(a,b)^{U_\theta}=\mathsf R(a,b)^{e}$
contains a $U/U_\theta$-cyclic weight vector  that actually
lies in $R(a_i,b_i)^\ah= R(a_i,b_i)^{e}(0)$, %i.e., $\g^{e}(0)$, 
This is exactly what we need.
\end{proof}

In view of this theorem and  question $(\diamondsuit)$ about $\g^e_{nil}$ %and $\g^e(1)$ 
for reachable elements, it is
desirable to know what reachable orbits are low-$\GR{A}{2}$ orbits. 
In \cite[Table\,25]{dy}, Dynkin pointed out all simple subalgebras of 
rank $>1$ in the exceptional 
algebras; in particular, the subalgebras of type $\GR{A}{2}$. (There are few errors in that
Table, which are corrected by Minchenko \cite[2.2]{minch}.) 
From this one easily deduces the list of $\GR{A}{2}$-pairs.
In the last column of Table~\ref{table_E}, we point out the $\GR{A}{2}$-pairs among 
all friendly pairs and thereby the low-$\GR{A}{2}$ orbits in the exceptional algebras.

\noindent
All reachable orbits among the orbits $G{\cdot}\edva$ are indicated in the first column
of Table~\ref{table_E}. 
However, this does not exhaust all reachable orbits in the exceptional algebras.
There are also reachable orbits that  are not included in a friendly pair,
Altogether, there still remain seven reachable orbits for $\GR{E}{8}$ and one orbit for each of
$\GR{E}{6}$, $\GR{E}{7}$, and $\GR{F}{4}$ that are not low-$\GR{A}{2}$ orbits. 

\begin{rmk}   \label{rem:low-c2}
One can say that $\co\subset\N$ is a {\it low-$\GR{C}{2}$ orbit\/} if there is a subalgebra
$\mathfrak{sp}_4\simeq \mathfrak{so}_5\subset \g$ such that $\co\cap\mathfrak{sp}_4$
is a minimal nilpotent orbit of $\mathfrak{sp}_4$. Such an orbit is not necessarily included in
a friendly pair, but it is always reachable. There is an analogue of Theorem~\ref{thm:reach}
for the low-$\GR{C}{2}$ orbits that can be derived from a description of the algebra 
$\bbk[Sp_4/U]^{U_\theta}$ and  the spaces $\mathsf R(a,b)^{U_\theta}$ for all 
simple $Sp_4$-modules $\mathsf R(a,b)$. Note that here $U_\theta\ne (U,U)$, hence this is
not related to \cite{odno-sv}.
However, the proof becomes much more involved, because %the algebra 
$\bbk[Sp_4/U]^{U_\theta}$ appears to be a hypersurface and we need an explicit description
of the unique relation. In the exceptional algebras, there are only two low-$\GR{C}{2}$ orbits that are not low-$\GR{A}{2}$ orbits (use again Dynkin's table!). These are orbits
$\GR{A}{3}+2\GR{A}{1}$ and $\GR{A}{2}+3\GR{A}{1}$ for $\g=\GR{E}{8}$. 
In view of such limited applicability, we do not include the proofs in this note. 
\end{rmk}

In the classical algebras, $\GR{A}{2}$-pairs correspond to representations of $\sltri$
(all, orthogonal, and symplectic, respectively). But this correspondence is not
bijective and it is not clear how to get a description of the corresponding partitions.

Theorem~\ref{thm:reach}, Remark~\ref{rem:low-c2}, and similar results for classical Lie algebras (see below)  strongly  support the following 

\begin{conj}   \label{gip:reach}  Let $\g$ be a simple Lie algebra.
If $e\in\N$ is reachable, then {\sf (a)}
$\g^{e}_{nil}$ is generated by $\g^{e}(1)$  and \ {\sf (b)}
$\g^{e}_{nil}\subset [\g^e, \g^e]$.
\end{conj}

For $\slv$, this is proved in \cite[Theorem\,4.5]{reach}. (Although  property (b) 
is not stated there, the argument actually proves both properties.)
The case of $\spv$ and $\sov$ is considered in \cite{osya}. Practically, we  have only 
eight unclear cases in exceptional Lie algebras. No doubt, this  can be verified using GAP.
But the challenge is, of course,  to find a conceptual proof.

%%%%%%%%%%%%%%%%%%
\section{Very friendly pairs of orbits}  
\label{sect:very}

\noindent
For a divisible orbit $\co=G{\cdot}e$ and an $\tri$-triple $\{e,h,f\}$, we agree to choose 
$\edva$  in $\g( 4)$, i.e., $[h,\edva]=4\edva$.

\begin{df}
A friendly pair $(\co, \co^{\langle 2\rangle})$ 
is said to be {\it very friendly\/}, if 
$[e,\edva]=0$ for a suitable choice of $\edva\in \g( 4)$, i.e., 
if $\co^{\langle 2\rangle}\cap \g^e( 4)\ne\varnothing$.
\end{df}

\begin{lm}   \label{commute}
If\/ $\g$ is a classical Lie algebra, then all friendly pairs are very friendly.
\end{lm}
\begin{proof}
The elements $\edva$ constructed in the  proof of Theorem~\ref{sl+sp+so}
commute with $e$.
\end{proof}
\noindent
%This is no longer valid for exceptional Lie algebras. Still, the following is true.

\begin{lm}    \label{lem:A2-very}
If $(G{\cdot}e, G{\cdot}\edva)$ is an $\GR{A}{2}$-pair, then it is very friendly.
\end{lm}
\begin{proof}
The property of being very friendly holds inside $\sltri$.
\end{proof}

\noindent This again shows that it is  helpful to know the $\GR{A}{2}$-pairs 
among pairs of orbits in Table~\ref{table_E}. 

\begin{thm}   \label{commute_exc}
If $\g$ is an exceptional Lie algebra and  $\ede$ is divisible,
then $(G{\cdot}e, G{\cdot}\edva)$ is very friendly,
%$G{\cdot}\edva\cap \g^e( 4)\ne\varnothing$, 
with only one exception---$G{\cdot}e$ being the orbit  $\GR{F}{4}(a_2)$ for\/ $\g=\GR{F}{4}$.
\end{thm}
\begin{proof}
1$^o$. Let us prove that all the pairs  in Table~\ref{table_E} are very friendly,
except the last pair for $\GR{F}{4}$. To this end, we employ the following technique:

\textbullet \ \  Combining Remark~\ref{es} and Lemma~\ref{commute} shows that
if $G{\cdot}e$ is divisible orbit,
$e$ lies in a {\sl classical\/} subalgebra $\es\subset\g$, and $S{\cdot}e$ is divisible,
then the pair in question is very friendly. By Theorem~\ref{thm:divis-Levi},
this applies to all orbits $G{\cdot}e$ in Table~\ref{table_E}
whose name  is  a (sum of) classical Cartan type(s).

\textbullet \ \  Even if a divisible orbit's name is an exceptional Cartan type, 
this orbit still can meet a regular\footnote{A subalgebra of $\g$ is called {\it regular\/} if it 
is normalised by a Cartan subalgebra} classical subalgebra that is not a Levi subalgebra. 
To see this, one has to use
Dynkin's tables \cite[Tables 16-20]{dy}, namely the column "minimal including regular 
subalgebras". For instance, the divisible $\GR{E}{8}$-orbit denoted nowadays by 
$\GR{E}{8}(b_6)$  has  the label  $\GR{D}{8}(a_3)$ in \cite[Table\,20]{dy}, which means 
that it is generated by a certain
distinguished orbit in $\GR{D}{8}=\mathfrak{so}_{16}$; 
actually, by the orbit corresponding to the partition $(9,7)$.
By Theorem~\ref{sl+sp+so}(ii), this $SO_{16}$-orbit is divisible.
Hence the corresponding pair of $\GR{E}{8}$-orbits is very friendly.
Similarly, the divisible $\GR{E}{8}$-orbit denoted nowadays by $\GR{E}{8}(a_6)$ has also 
the label $\GR{A}{8}$. This means that it is generated by the principal 
nilpotent orbit in $\mathfrak{sl}_9$, which is divisible. 
Such an argument also applies to the orbits
$\GR{G}{2}(a_1)$, $\GR{F}{4}(a_3)$, $\GR{E}{8}(a_7)$. 

\textbullet \ \  Finally, in view of  Lemma~\ref{lem:A2-very}, all $\GR{A}{2}$-pairs are 
very friendly.

\noindent
After all these considerations, only three divisible orbits left: 
%namely, the last items
$\GR{F}{4}(a_2)$ for $\g=\GR{F}{4}$; \ $\GR{E}{6}(a_1)$ for $\g=\GR{E}{6}$; \ 
$\GR{E}{8}(a_4)$ for $\g=\GR{E}{8}$. In the last two cases, we can show via direct bulky 
considerations that the pairs are very friendly, while the first case represents the only 
non-very friendly pair. Below, we consider in details this bad case.

2$^o$. In this part of the proof, $\g$ is a simple Lie algebra of type $\GR{F}{4}$.
The orbit $\GR{F}{4}(a_2)$ is distinguished
and $\dim\g^e( 4)=1$. 
Therefore, it suffice to test a non-zero element of $\g^e( 4)$.
We will prove that the height of such a non-zero element 
is strictly less than $\htt(e)/2$. 
The numbering of the simple roots of simple Lie algebras follows
\cite{VO}, and the $i$-th fundamental weight  
is denoted by $\varpi_i$.

There is an involutory automorphism $\vartheta$ of $\g$, with 
the corresponding $\BZ_2$-grading $\g=\g_0\oplus\g_1$, such that 
the subalgebra $\g_0$ is of type $\GR{C}{3}+\GR{A}{1}$.
If $\co_0$ is the regular nilpotent orbit in $\g_0$, then
$G{\cdot}\co_0$ is the orbit $\GR{F}{4}(a_2)$ in $\g$. 
This can be verified as follows.
The $\g_0$-module $\g_1$ is isomorphic to $\mathsf R(\varpi_3)\otimes \mathsf R_1$.
(Here $\mathsf R(\varpi_3)$ is a 14-dimensional $\GR{C}{3}$-module and $\mathsf R_1$
is the standard two-dimensional $\GR{A}{1}$-module.)
Let $\ah$ be a principal $\tri$ in $\g_0$. Decomposing $\g_0$ and $\g_1$ as
$\ah$-modules, one obtains 
\begin{equation}  \label{eq:decomp0_1}
  \g_0=2\mathsf R_2+\mathsf R_6+\mathsf R_{10} \quad \text{ and } \quad 
  \g_1=\mathsf R_2+\mathsf R_4+\mathsf R_8+\mathsf R_{10} ,
\end{equation}
where $\mathsf R_n$ stands for the $(n{+}1)$-dimensional simple $\ah$-module.
From \eqref{eq:decomp0_1}, it follows that if $e$ is a nonzero nilpotent element of 
$\ah$, then $\dim \g^e=8$. 
Hence $\dim G{\cdot}e=44$ and it is the
orbit $\GR{F}{4}(a_2)$, as claimed. (The  algebra of type $\GR{F}{4}$ has 
a unique nilpotent orbit of dimension $44$.)
By \eqref{eq:decomp0_1},
the unique 5-dimensional simple $\ah$-module $\mathsf R_4$
occurs in $\g_1$. This means that, for $e\in \ah\subset\g_0$, the subspace
$\g^e( 4) $ lies in $\g_1$. 
To get a precise description of $\g( 4)\cap\g_1$, we
use an explicit model of $\g_0$ inside $\g$. 
Let $\te$ be a common Cartan subalgebra of $\g$ and $\g_0$ and
let $\ap_1,\dots,\ap_4$ be the simple roots of $(\g,\te)$. Then $\ap_1,\ap_2,\ap_3$
are the simple roots of $\GR{C}{3}$ and $\theta=2\ap_1+4\ap_2+3\ap_3+2\ap_4$
is the simple root of $\GR{A}{1}$. (Note that $\theta$ is the highest root for 
$\g$.) The roots of $\g_1$ are those having the coefficient of $\ap_4$ equal
to $\pm 1$.

We assume that $e=e_{\ap_1}+e_{\ap_2}+e_{\ap_3}+e_{\theta}$ and
$h\in\te$ is the standard characteristic of $e\subset \co_0$, 
i.e., $\ap_i(h)=2$, 
$i=1,2,3$, and $\theta(h)=2$. Then $\ap_4(h)=-8$. Consider the $\BZ$-grading of $\g$
determined by $h$. 
Using the above values $\ap_i(h)$, one easily 
finds that $\dim(\g_1\cap \g( 4))=3$ and the corresponding roots are
\[
   \nu_1=\ap_1+3\ap_2+2\ap_3+\ap_4, \ \nu_2=2\ap_1+2\ap_2+2\ap_3+\ap_4, \  
   \nu_3=-\ap_2-\ap_3-\ap_4 .
\]
That is, the 1-dimensional subspace $\g^e( 4) $ lies in 
$\g_{\nu_1}\oplus\g_{\nu_2}\oplus\g_{\nu_3}$. Next, the list of roots of $\GR{F}{4}$ shows that
$\ad e$ takes $\g_{\nu_1}\oplus\g_{\nu_2}$ to the 1-dimensional space $\g_{\mu}$,
where $\mu=2\ap_1+3\ap_2+2\ap_3+\ap_4$. Therefore, $\g^e( 4)$  must belong to 
$\g_{\nu_1}\oplus\g_{\nu_2}$.
Since $\hot(\nu_1)=\hot(\nu_2)=7$ and $-11\le \hot(\gamma)\le 11$ for any $\gamma\in
\Delta(\GR{F}{4})$, we see that
$\htt(x)\le 3$ for all $x\in \g_{\nu_1}\oplus\g_{\nu_2}$.
Since $\htt(e)=10$, and hence $\htt(\edva)=5$,
the orbit  $G{\cdot}\edva$ cannot meet the 1-dimensional subspace $\g^e( 4)
\subset \g_{\nu_1}\oplus\g_{\nu_2} $.
\end{proof}

\end{document}